\documentclass[12pt]{article}
\usepackage{latexsym,amsmath,amsthm, amssymb}

\usepackage{a4wide}

\newtheorem{theorem}{Theorem}

\newtheorem{lemma}[theorem]{Lemma}

\theoremstyle{remark}

\def\hess{\mathop{\rm Hess}\nolimits}
\def\tr{\mathop{\rm Tr}\nolimits}
\def\E{\mathop{\rm Ent}\nolimits}
\def\lg{\mathop{\rm log}\nolimits}
\def\d{\mathop{\rm div}\nolimits}
\def\R{\mathbb R}

\usepackage{a4wide}

\def\longto{\longrightarrow}

\numberwithin{equation}{section}

\begin{document}


\title{Entropy jumps for isotropic log-concave random vectors and spectral gap}
\author{Keith Ball \\
Institute of Mathematics, University of Warwick,\\
Coventry, CV4 7AL, UK\\
Email: kmb120205@googlemail.com
\and
Van Hoang Nguyen \\
Institut de Math\'ematiques de Jussieu, UPMC, \\
4 place Jussieu, 75252 Paris, France\\
Email: vanhoang@math.jussieu.fr }

\maketitle


\renewcommand{\thefootnote}{}

\footnote{2010 \emph{Mathematics Subject Classification}: 94A17.}

\footnote{\emph{Key words and phrases}: Entropy, Fisher information, isotropic constant, isotropic log-concave random vector, spectral gap.}

\renewcommand{\thefootnote}{\arabic{footnote}}
\setcounter{footnote}{0}

\begin{abstract}
We prove a quantitative dimension-free bound in the Shannon-Stam Entropy inequality for the convolution of two log-concave distributions in dimension $d$ in terms of the spectral gap of the density. The method relies on the analysis of the Fisher  Information production, which is the second derivative of the Entropy along the (normalized) Heat semi-group.
We also discuss consequences of our result in the study of the isotropic constant of log-concave distributions (slicing problem).
\end{abstract}

\section{Introduction}
Let $X$ be a random vector in $\R^d$ with density $f:\R^d\longrightarrow  [0,\infty)$, a relation denoted by $X\sim f$. Its entropy is defined to be
$$\E(X)=-\int_{\R^d}f\lg f$$
provided $\int_{\R^d}f\lg_+ f < \infty$. We then say that $X$ has finite entropy.

We shall say that a random vector $X$ on $\R^d$, or a probability density $f$, is 
 \emph{isotropic} if it is centered and has a  covariance matrix equal to the identity:
 $$\mathbb E [X] = \int_{\R^d} x \, f(x)\, dx = 0$$ 
and 
$$ \mathbb E [X_i X_j] = \int_{\R^d} x_i x_j \, f(x) \, dx = \delta_{i,j}, \quad i,j=1\ldots , d .$$
This normalization can be realized by an affine transformation. 
 
 Among random vectors with a given covariance matrix, the corresponding Gaussian has the largest entropy. 
 The gap between the entropy of a random vector $X$ and that of a Gaussian with same covariance matrix is a strong measure of how close $X$ is to being Gaussian. 
 For instance, if $X$ has  mean zero and is isotropic with density $f$, and if $G$ is a standard (normal) Gaussian vector with density $g$, then the Pinsker-Csisz\'ar-Kullback inequality (see \cite{P, C} or \cite{BL}) implies that
 $$\frac{1}{2}\bigl{(}\int_{\R^d}|f-g|\bigl{)}^2\leq \E(G)-\E(X).$$
 
The celebrated Shannon-Stam inequality (see \cite{SW, S}) says that if $X$ and $Y$ are independent identically distributed (\emph{iid} in short) random vectors, then the normalized sum $(X+Y)/\sqrt{2}$ has entropy at least as large as that of $X$ and $Y$:
\begin{equation*}\label{ST}
 \E\Big(\frac{X+Y}{\sqrt{2}}\Big)\ge\E(X).
 \end{equation*}
Moreover,  this inequality is strict if $X$ is not itself a Gaussian random vector. A challenging problem is to quantify this phenomenon, i.e. for fixed covariance matrix (say the identity), bound $ \E(\frac{X+Y}{\sqrt{2}})-\E(X)$ from below by a non-negative (and positive outside zero) function of $\E(G)-\E(X)$. The first result in this direction was obtained by Carlen and Soffer~\cite{CS} who proved, under technical assumptions, a non-explicit bound based on a compactness argument.  Extra assumptions cannot be avoided if one aims at universal entropic estimates: it is  easy  to construct (by taking a carefully chosen double bumped Gaussian) a random variable $X$ for which, the convolution does not greatly modify the entropy, $\E(\frac{X+Y}{\sqrt{2}}) \simeq \E(X)$, but with $\E(X) \ll \E(G)$.
A surprisingly neat result holds in the case where $X\in \R$ is a random variable with variance $1$ and with a density $f$ that satisfies a Poincar\'e (or spectral gap) inequality in the sense that for some positive $c$ and
any smooth function $s$ with $\int_{\R}fs=0$ $$  c\int_{\R}f \, s^2\leq \int_{\R}f\, \big(s'\big)^2.$$
Indeed, we then have
 \begin{equation}\label{bbn}
 \E\Big(\frac{X+Y}{\sqrt{2}}\Big)-\E(X)\geq \frac{c}{2+2c}\big(\E(G)-\E(X)\big)
 \end{equation}
for $Y$  an independent copy of $X$. This result was proved by Ball, Barthe and Naor in~\cite{BBN} using a variational formula for the Fisher information of a marginal density and spectral analysis to get an information jump in the presence of a spectral gap (see \cite[Theorem 2]{BBN}) and then  using a relation between the Fisher information and entropy provided by the adjoint Ornstein-Uhlenbeck semigroup. In the paper~\cite{BJ}, Barron and Johnson obtained  a result similar to~\eqref{bbn} under the same hypothesis, but their method is different (at least in details) to that of \cite{BBN}. In their paper, Barron and Johnson give an upper bound for the $L^2$ distance of the score function of $(X+Y)/\sqrt{2}$ to the space of additive functions of $X$ and $Y$, using $L^2$-orthogonal projections; they eventually use the Poincar\'e inequality to get the result (see \cite[Proposition 2.1 and 3.1]{BJ} for details).
 
 The aim of the present paper is to investigate similar results for random vectors, and incidentally to give a new approach to~\eqref{bbn}. A random vector $X\in \R^d$ with density $f$ is said to satisfy a Poincar\'e or spectral gap inequality with constant $c>0$ if for any smooth function $s$ with $\int_{\R^d}fs=0$
 \begin{equation}\label{spectralgap}
  c\int_{\R^d}f s^2\leq \int_{\R^d}f\, |\nabla s|^2.
 \end{equation}
That largest constant $c$ in this inequality is indeed the spectral gap for the operator $-L$ on $L^2(f)$ where $Ls:= \Delta s - \nabla (\log(f)\cdot \nabla s$ for suitable functions; the Poincar\'e constant ${\rm c_p}(f)$  refers rather to the inverse of the spectral gap, i.e. to the smallest constant in the inequality 
  \begin{equation}\label{spectralgap2}
  \int_{\R^d}f s^2\leq  {\rm c_p}(f)\int_{\R^d}f\, |\nabla s|^2.
 \end{equation}
 
 A simplistic adaptation of the argument used~\cite{BBN} in higher dimensions leads to an inequality of the form~\eqref{bbn} for random vectors  but with an extra dependance in $d$, the dimension. In the present paper we prove the result without the extra dependence for log-concave random vectors, i.e. those having a density $f$ such that $-\log(f)$ is convex on $\R^n$. It is well known that such random vectors have finite entropy and have a positive spectral gap (see below). This family is central in many high-dimensional problems.

\begin{theorem}\label{main}
Let $X$ be an isotropic  log-concave random vector in $\R^d$. Assume its density $f$ satisfies a Poincar\'e inequality~\eqref{spectralgap} with constant $c>0$.
Then, if $Y$ is an independent copy of $X$, we have
\begin{equation}
\label{vectorjump}
\E\left(\frac{X+Y}{\sqrt{2}}\right)-\E(X)\geq \frac{c}{4(1+c)}(\E(G)-\E(X)).
\end{equation}
\end{theorem} 
Since $c\le 1$ when $f$ is isotropic,  the constant $\frac{c}{4(1+c)}$ may be replaced by $\frac{c}{8}$. The log-concavity assumption will be crucial on the proof of inequality~\eqref{marg1bis} below, we do not know whether it holds without this assumption.

As mentioned above, we need to develop a method different to the one of~\cite{BBN}. Our alternative approach relies on the study of \emph{second} derivatives of the entropy  along the heat semi-group (or rather along the Ornstein-Uhlenbeck semi-group). Note that it also gives an alternative proof of the one dimensional case~\eqref{bbn}, up to a numerical (nonessential) constant.

Let us make some general comments on  log-concave random vectors. If $X$ and $Y$ are \emph{i.i.d} random vectors with density $f$, the normalized sum $\frac{X+Y}{\sqrt{2}}$ has density
\begin{equation}\label{convol}
u\longto \int_{\R^d}f(\frac{u+v}{\sqrt{2}})f(\frac{u-v}{\sqrt{2}})dv
\end{equation}
which is a marginal of the joint density on $\R^{2d}$ of the pair $(X,Y)$. It is a consequence of the Brunn-Minkowski inequality (in its functional form due to Pr\'ekopa~\cite{Prekopa}) that log-concave random vectors have log-concave marginals and hence that if $X$ and $Y$ are log-concave, then so is $\frac{X+Y}{\sqrt{2}}$. It is also well known that a log-concave density verifies a Poincar\'e inequality for some constant $c>0$. It was proven by Kannan, Lovasz and Simonovitz~\cite{KLS} and independently by Bobkov~\cite{Bo} that if $X$ is an isotropic log-concave random vector, then it satifies a Poincar\'e inequality~\eqref{spectralgap} with consant $\frac{C}{n}$ for some numerical constant $C$ (thus independent of $n$ and $X$). Actually, slightly better depedence in $n$ is known.The Kannan-Lovasz-Simonovitz (KLS) open conjecture states that there exists a universal constant $c>0$ such that for every $n$, every every isotropic log-concave random vector in $\R^n$ verifies a Poincar\'e inequality with constant $c$. It was noticed some time ago by the first named author, as part of a general program of understanding information theory (and entropy) in the context of convexity in high dimensions,  that using~\eqref{vectorjump} we can prove that the KLS conjecture implies the celebrated hyperplane (or slicing) conjecture.  We shall return to this in the last section. 

The organization of the paper is as follows. In the next section, we recall some standard facts about the Ornstein-Uhlenbeck semi-group and about the first derivative of entropy (Fisher information) and the second derivative (information production). Next we investigate how information production behaves under convolution and state a general inequality relating information production of a random vector to the information of a marginal. The subsequent section contains the proof of Theorem~\ref{main}. The final section discusses the connections between entropy jump and the isotropic constant of  log-concave distribution.

\section{Classical results on Ornstein-Uhlenbeck semi-group and Fisher information}
\label{sect:OU}

For any random vector $X$ with smooth enough density $f$ -we require that $\sqrt{f} \in H^1(\R^n)$, but later the density will have even smoother behavior- its Fisher information is defined by
$$J(X):=J(f):=\int_{\R^d}\frac{|\nabla f|^2}{f}.$$

Among random vectors with  given covariance matrix, the Gaussian has the smallest Fisher information, as shown by the following straightforward computation:
if $X$ is an isotropic mean-zero  log-concave random with density $f$, and $G$ is a standard Gaussian with density $g: x\to (1/\sqrt{2\pi})^de^{-|x|^2/2}$, then
$$J(G)=\int_{\R^d}\frac{|\nabla g|^2}{g}=\frac{1}{\sqrt{2\pi}^d}\int_{\R^d}\sum_{i=1}^dx_i^2e^{-|x|^2/2}dx=d.$$
and, by integration by parts,
$$
0\leq\int_{\R^d}\Bigl{|}\frac{\nabla f}{f}+x\Bigl{|}^2 f(x)dx=J(X)-2\int_{\R^d} \textrm{div}(x) f  + d =  J(X)-J(G).
$$

The Fisher information appears as the derivative of the entropy  along the Ornstein-Uhlenbeck semi-group, a property central in the works by Bakry and  \'Emery \cite{BE} and also in Barron's work \cite{B} on the convergence of entropy in the central limit theorem. The Ornstein-Uhlenbeck semi-group can be constructed in several (equivalent) ways and we choose the following. If $X$ is a random vector with density $f$ and $G$ is a standard Gaussian, independent of $X$, we consider the random vector  $X_t=e^{-t}X+\sqrt{1-e^{-2t}}G$, whose law is the Ornstein-Uhlenbeck evolute at time $t$ of the law of $X$. More precisely, the density $f_t$ of $X_t$ is the solution of the  Fokker-Planck equation with generator $L$ :
 \begin{equation}\label{f-p}
 f_0 = f ; \qquad \frac{\partial}{\partial t} f_t(x) =L(f_t)(x) := \Delta_x f_t+\d_x(xf_t), 
 \end{equation}
for all $t>0$ and $x\in \R^n$.

It is indeed well known, that starting with a (continuous, say) density $f$, the density $f_t$ is, for $t>0$, strictly positive, $C^\infty$-smooth on $\R^n$ and that $f_t$ and its derivatives decay  exponentially fast to zero at $\infty$ ; in particular $f_t$ has a finite Fisher information and it is readily checked that
\begin{align*}
\frac{\partial}{\partial_t}\E(f_t)&=-\int_{\R^d}(Lf_t)\log(f_t) = -\int_{\R^d} f_t \Delta \log(f_t) -d\int_{\R^d}f_t\\
&=J(f_t)-d.
\end{align*}
Hence, we have the classical expression of the entropy gap as the integral of the information gap
\begin{equation}\label{EFI}
\E(G)-\E(X)=\int_0^{\infty}(J(f_t)-d)dt.
\end{equation}
We refer to Carlen and Soffer \cite{CS} for details and precise justifications. 

Let us mention for further reference some other nice stability properties of the Ornstein-Uhlenbeck semi-group.
It can only improve the spectral gap: if $X\sim f$ is isotropic and satisfies a Poincar\'e inequality~\eqref{spectralgap} with constant $c>0$ (actually, $c\in ]0,1]$), then $f_t$ satisfies a Poincar\'e inequality with the same constant $c>0$. This follows easily from Fubini's theorem, H\"older's inequality and the fact that the Gaussian density has spectral gap of size $1$ (see~\cite{BBN}). 
Next,  it is again a consequence of Pr\'ekopa's theorem that if $X$ (or $f$) is log-concave, then so is $X_t$ (or $f_t$). Finally, it is also classical that the operation of taking marginals commutes with the Ornstein-Uhlenbeck semi-group in the following sense. Let $X$ and $Y$ be two independent random vectors and $X_t=e^{-t}X+\sqrt{1-e^{-2t}}G_1$ and $Y_t=e^{-t}Y+\sqrt{1-e^{-2t}}G_2$ their independent evolutes along the Ornstein-Uhlenbeck semi-group, where each $G_i$ $(i=1,2$) is a  standard Gaussian vector independent of all the other vectors. Then
\begin{equation}\label{OUmarg}
\frac{X_t+Y_t}{\sqrt2 }= e^{-t}\, \frac{X+Y}{\sqrt2} \;+\; \sqrt{1-e^{-2t}}\, G
\end{equation}
where $G=\frac{G_1+G_2}{\sqrt2}$ is a standard Gaussian vector.



Throughout the rest of the section, $X$ will be an isotropic  log-concave random vector with mean zero,  and density $f$. The density $f_t$ of $X_t=e^{-t}X+\sqrt{1-e^{-2t}}G$, the evolute of $X$ along the Ornstein-Uhlenbeck semi-group satisfies~\eqref{f-p} and takes the form $f_t=e^{-\varphi_t}$ with  $\varphi_t:=\log(f_t)$ convex on $\R^d$. Its  Fisher information will be denoted by
\begin{equation}\label{defJt}
J(t):=J(f_t)=\int_{\R^d}\frac{|\nabla f_t|^2}{f_t}=-\int_{\R^d} f_t\Delta\lg f_t=\tr\int_{\R^d}f_t\hess\varphi_t.
\end{equation}
We will  work with the derivative $\partial_t J(t)$ of the Fisher information along the Ornstein-Uhlenbeck semi-group. The following result is classical in the context of Bakry-Emery's $\Gamma_2$ calculus, although it is not usually written in this form which for us will prove useful. We include a proof for completeness.

\begin{lemma}~\label{lem:dtJ}
With the previous notation we have
\begin{equation}\label{derJ}
\partial_t J(t)=2J(t)-2\tr\int_{\R^d} f_t\, \big(\hess\varphi_t\big)^2.
\end{equation}
\end{lemma}

\begin{proof}
Denoting $\partial_j$ the partial derivative (in space) with respect to $x_j$ we have
\begin{align*}
\partial_t J(t)&=\sum_{i=1}^d\int_{\R^d}2\frac{\partial_if_t}{f_t}\partial_i\partial_tf_t-\int_{\R^d}\big(\frac{\partial_if_t}{f_t}\big)^2\partial_tf_t\\
&=\sum_{i=1}^d\int_{\R^d}-2\partial_i\varphi_t\partial_i\Big[\sum_{j=1}^d\partial_j\big((-\partial_j\varphi_t+x_j) f_t\big) \Big]\\
&\ \ \ \ -\int_{\R^d}\sum_{j=1}^d\partial_j\big((-\partial_j\varphi_t+x_j) f_t \big) (\partial_i\varphi_t)^2.
\end{align*}
where we used that $f_t$ follows~\eqref{f-p} and $Lg = \sum_j \partial_j \big(\partial_j g + x_j g \big)$.
Let $A$ and $B$ be the first and the second terms in the above sum, respectively. Then, we have, by integration by parts and~\eqref{defJt},
\begin{align*}
A&=2\sum_{i,j=1}^d\int_{\R^d}(\partial_{ij}\varphi_t)\partial_i\big((-\partial_j\varphi_t+x_j)f_t\big)\\
&=2J(t)-2\tr\int_{\R^d}f_t\, \big(\hess\varphi_t\big)^2+2\sum_{i,j=1}^d\int_{\R^d}f_t\, \big(\partial_{ij}\varphi_t\big)\big(\partial_i\varphi_t\big)\big(\partial_j\varphi_t-x_j\big)
\end{align*}
and
\begin{align*}
B&=-\sum_{i=1}^d\int_{\R^d}\sum_{j=1}^d\partial_j(f_t(-\partial_j\varphi_t+x_j)) (\partial_i\varphi_t)^2\\
&=2\sum_{i,j=1}^d\int_{\R^d}f_t(-\partial_j\varphi_t+x_j)(\partial_i\varphi_t)(\partial_{ij}\varphi_t)\\
&=-2\sum_{i,j=1}^d\int_{\R^d}f_t (\partial_{ij}\varphi_t) (\partial_i\varphi_t) (\partial_j\varphi_t-x_j).
\end{align*}
Taking the sum of $A$ and $B$, one gets the result of this lemma.
\end{proof}

Note that the formula in the previous lemma can be rewritten in the following equivalent form, which is more standard:
\begin{eqnarray} 
\partial_t \big(J(t) - d \big) &=& -2 \big(J(t)-d\big) -2\tr\int_{\R^d} f_t\, \big(\hess\varphi_t- \textrm{Id}\big)^2\nonumber\\
&\le&  -2 \big(J(t)-d\big). \label{Gamma2}
\end{eqnarray}

The next lemma will allow us to control the tails of the entropy production.

\begin{lemma}\label{lem:intJ}
With the previous notation we have
$$\E(G)-\E(X)\leq 2\int_0^{\infty}e^{-2t}(J(t)-d)dt.$$
\end{lemma}
\begin{proof}
Integration of inequality ~\eqref{Gamma2} leads to the following classical Gaussian Log-Sobolev inequality
$$J(t)-d \geq 2(\E(G)-\E(X_t))\text{ }\forall t > 0.$$
By integration by parts, we get
$$\int_0^{\infty}e^{-2t}(J(t)-d)dt\geq \E(G)-\E(X)-\int_0^{\infty}e^{-2t}(J(t)-d)dt,$$
or equivalentely
$$\E(G)-\E(X)\leq 2\int_0^{\infty}e^{-2t}(J(t)-d)dt.$$
\end{proof}

\section{A result for the information production of marginals}

As we saw in ~\eqref{derJ}, the information production $ \partial_t J(t)$ along the Ornstein-Uhlenbeck semi-group is given by quantities of the form
$$\tr\int \big(\hess\log f\big)^2 f  .$$
For our argument, we need to analyze how such quantities can be estimated for marginal densities. Assume $Z$ is a random vector with density $\omega:\R^N\to \R_+$ ($N\ge 1$) and consider the projection $P_E Z$ of $Z$ onto a subspace $E\subset \R^N$. It has a density on $E\simeq \R^{\textrm{dim}(E)}$ which we denote by $h$. A useful observation due to Carlen~\cite{Ca} for Fisher information is that
$$J(h) \le \int \frac{|P_E\nabla \omega|^2}{\omega}.$$
The next result provides an analogue for  information production. However, we are able to establish it only in the case of log-concave densities: here is where the restriction in  our main theorem comes from. Using it, we can then state the central inequality that will be used in the proof of the main theorem.

\begin{lemma}\label{margentprod}
Let $N\ge 1$ and  $\omega=e^{-\phi}:\R^N \to \R_+$ be a smooth positive function. Given a subspace $E\subset \R^N$ define the marginal function on $E$ by
$$\forall x \in E, \qquad h(x) := e^{-\psi(x)} := \int_{E^\perp} \omega(x+y) \, dy=  \int_{E^\perp}e^{-\phi(x+y)} \, dy.$$
Denote by $P_E$ the orthogonal projection onto $E$.
Then, for every $x\in E$ we have,
$$h(x) \hess \psi(x) \le
 \int_{E^\perp}  \omega(x+y)\, P_E \hess \phi (x+y) P_E  \,dy 
$$
 in the operator sense (for symmetric operators on $E$) and if $\hess \psi(x) \ge 0$ , then
\begin{equation}\label{marg1}
\tr\left[ \big(\hess \psi(x)\big)^2  h(x)  \right] \le\int_{E^\perp}   \tr\big[\big(P_E \hess \phi (x+y) P_E \big)^2\big]  \omega(x+y) \, dy.
\end{equation}
Therefore, if $\hess\psi\ge 0$ we have 
\begin{equation}\label{marg1bis}
 \int_E  \tr\left[\big(\hess \log h\big)^2\right]  h \le \int_{\R^N}   \tr\left[\big(P_E (\hess \log \omega) P_E \big)^2\right]  \omega .
\end{equation}
\end{lemma}

\begin{proof}
We start with the observation that for $x\in E$,
\begin{equation}\label{CS1}
\frac{\nabla h(x) \otimes \nabla h(x)}{h(x)} \le  \int_{E^\perp} \frac{P_E \nabla \omega(x+y) \otimes P_E\nabla \omega(x+y)}{\omega(x+y)} \, dy
\end{equation}
in the symmetric operator sense (on $E$). Indeed, we have
$$\nabla h(x) = \int_{E^\perp} P_E \nabla \omega(x+y) \, dy$$ 
and for any $v \in E$ we have, using the Cauchy-Schwartz inequality
$$
(\nabla h(x) , v)^2  \le \int_{E^\perp} \frac{( \nabla \omega(x+y) , v )^2}{\omega(x+y)} \, dy \int_{E^\perp} \omega(x+y) \, dy 
$$
as claimed.
Next, observe that 
$$\hess h(x) = \int_{E_\perp} P_E \hess \omega (x+y) P_E  \, dy $$
and
$$h(x) \hess \psi(x) = h(x) \hess (-\log h)(x) = \frac{\nabla h(x) \otimes \nabla h(x)}{h(x)} - \hess h (x) .$$
Thus~\eqref{CS1} leads to the inequality
\begin{align*}
h(x) \hess \psi(x)  &\le \int_{E^\perp} \Big(\frac{P_E \nabla \omega(x+y) \otimes P_E\nabla \omega(x+y)}{\omega(x+y)}  - 
 P_E \hess \omega (x+y) P_E  \Big) \, dy \\
 &= \int_{E^\perp}  \omega(x+y)\, P_E \hess \phi (x+y) P_E dy
 \end{align*}
 in the operator sense on $E$, as wanted. 
 
 Using that  that for symmetric operators $A\ge B \Rightarrow \text{Tr}(AH) \ge \text{Tr}(BH)$ whenever $H\ge 0$, we deduce that, when $\hess \psi (x) \ge 0$, 
$$ \tr\left[ \big(\hess \psi(x)\big)^2  h(x)  \right] \le\int_{E^\perp}   \tr\big[P_E \hess \phi (x+y) P_E \hess \psi(x) \big]  \omega(x+y) \, dy. $$
By the Cauchy-Schwartz inequality  (in vectorial form, for the Hilbert-Schmidt scalar product) we then have
\begin{align*}
 \tr\left[ \big(\hess \psi(x)\big)^2  h(x)  \right]  & \le  \sqrt{
 \int_{E^\perp}   \tr\big[\big(P_E \hess \phi (x+y) P_E  \big)^2 \big] \omega(x+y) \, dy. }\\
 &\qquad \times  
\sqrt{ \int_{E^\perp}   \tr\big[\big(\hess \psi(x) \big)^2\big]  \omega(x+y) \, dy. } 
\end{align*}
Noting that  the second integral equals $ \tr\left[ \big(\hess \psi(x)\big)^2  h(x)  \right]  $, we arrive to inequality~\eqref{marg1}. Integration over $E$ then gives~\eqref{marg1bis}.
\end{proof}

For our argument, we will need the following  useful observation.

\begin{theorem}\label{K2}
Let $X$ be a log-concave random vector in $\R^d$ with smooth density $f=e^{-\varphi}$ where $\varphi$ is a convex function on $\R^d$, and let $Y$ be an independent copy of $X$. Denote by  $h=e^{-\psi}$ the density on $\R^d$ of the random vector
 $\frac{X+Y}{\sqrt{2}}$ and put
$$K=\tr\Big[\int_{\R^d} \big(\hess \varphi\big)^2\,  f \Big], \quad K_2=\tr\Big[\int_{\R^d} \big(\hess\psi\big)^2 h\Big],$$
and
$$ M=\tr\Big[\Big(\int_{\R^d} (\hess \varphi )\, f \Big)^2\Big].$$   
Then, we have
$$K_2\leq \frac{K+M}{2}.$$
\end{theorem}
\begin{proof}
As mentioned earlier, we know by Pr\'{e}kopa's theorem that $h$ is log-concave, i.e. $\hess \psi \ge 0$.

We denote $\omega(x,y)=f(x)f(y)$ the density of $(X,Y)$ on $\R^d\times \R^d=\R^{2d}$. For $i=1, \ldots d$ we set $e_i=(0,...,\frac{1}{\sqrt{2}},0,\ldots, 0,\frac{1}{\sqrt{2}},0, \ldots)$ where the $i$-th and $(d+i)$-th coordinates are equal to $\frac{1}{\sqrt{2}}$ and the others are zero. Let $E$ be the vector subspace of $\R^{2d}$ spanned by the orthogonal family $\{e_1,\cdots,e_d\}$.We can assume that the density $h=e^{-\psi}$ of  the random vector $\frac{X+Y}{\sqrt 2}$ is defined on $E$ by identification of $\R^d$ and $E$ through the orthonormal basis $\{e_i\}$ of $E$. Then the Lemma~\ref{margentprod} gives
\begin{align*}
K_2&\leq 
 \int_{\R^{2d}}\omega(x,y)\tr\left[ \Big[ \big(\hess(-\log\omega)(x,y)e_i,e_j\big)\Big]_{i,j}^2 \right]\, dx dy \\
&=\int_{\R^{2d}}f(x)f(y)\sum_{i,j=1}^d \big(\hess(-\log\omega)(x,y)e_i,e_j\big)^2\, dxdy.
\end{align*}
Direct computation gives
$$ \big(\hess(-\log\omega)(x,y)e_i,e_j\big)^2=\frac{1}{4}(\partial_{ji}\varphi(x)+\partial_{ji}\varphi(y))^2$$
and hence
\begin{align*}
K_2&\leq \sum_{i,j=1}^d\frac{1}{4}\int_{\R^{2d}}f(x)f(y)\big((\partial_{ji}\varphi(x))^2+2\partial_{ji}\varphi(x)\partial_{ji}\varphi(y)+(\partial_{ji}\varphi(y))^2\big)\, dxdy\\
&=\frac{1}{2}\sum_{i,j=1}^d\int_{\R^{d}}f(\partial_{ji}\varphi)^2+\frac{1}{2}\sum_{i,j=1}^d\biggl(\int_{\R^d}f\partial_{ji}\varphi\biggl)^2\\
&=\frac{1}{2}(K+M). 
\end{align*}
\end{proof}

\section{Proof of Theorem 1.1}
We go back to the situation and the notation of Section~\S\ref{sect:OU}. $X$ is an isotropic  log-concave random vector with mean zero,  and density $f$, and $X_t$ is its evolute along the  Ornstein-Uhlenbeck semi-group.
The (log-concave) density of $X_t$ is denoted by $f_t=e^{-\varphi_t}$ and we set
$$
J(t):=J(X_t)=\tr\int_{\R^d}f_t\hess\varphi_t
$$
and
$$
K(t):=\tr\int_{\R^d}f_t(\hess\varphi_t)^2=-\frac{1}{2}e^{2t}\partial_t(e^{-2t}J(t))
$$
where we used Lemma~\ref{lem:dtJ} for the last equality.

We now consider $Z_t$, the Ornstein-Uhlenbeck evolute of $\frac{X+Y}{\sqrt 2}$ where $Y$ is an independent copy of $X$. As mentioned earlier~\eqref{OUmarg}, $Z_t = \frac{X_t+Y_t}{\sqrt 2}$ where $Y_t$ is an Ornstein-Uhlenbeck evolute of $Y$ independent of $X_t$. Denote by $h_t=e^{-\psi_t}$ the smooth (log-concave) density of $Z_t$ and set accordingly
$$
J_2(t):=J(Z_t)=\tr\int_{\R^d}h_t\hess\psi_t
$$
and
$$
K_2(t):=\tr\int_{\R^d}h_t(\hess\psi_t)^2=-\frac{1}{2}e^{2t}\partial_t(e^{-2t}J_2(t))
$$

Theorem~\ref{K2} applied to $X_t$ and $Z_t=\frac{X_t+Y_t}{\sqrt 2}$ then gives that
$$K_2(t)\leq \frac{K(t)+M(t)}{2}=K(t)-\frac{K(t)-M(t)}{2}$$
where,
$$M(t):=\tr\Big[\Big(\int_{\R^d} (\hess \varphi_t )\, f_t \Big)^2\Big].$$
This can be rewritten as
\begin{equation}\label{5}
\partial_t(e^{-2t}(J_2(t)-J(t)))\geq e^{-2t}(K(t)-M(t))
\end{equation}

We next claim that
\begin{equation}\label{spectK}
K(t)-M(t)\geq \frac{c}{1+c}(K(t)-J(t))
\end{equation}
To prove this, remember, as recalled in Section~\S\ref{sect:OU}, that $f_t$ verifies Poincar\'e's inequality with the  same (or better) constant $c$ as $f$. We apply the Poincar\'e inequality~\eqref{spectralgap} to the density $f_t=e^{-\varphi_t}$ and to the functions
$$s_i(x)=\partial_i\varphi_t (x)-\sum_{j=1}^dx_j\int_{\R^d}\big(\partial_{ij}\varphi_t \big) f_t$$
which verify that $\int s_i \, f_t=0$, for $i=1, \ldots, d$. After summing the inequalities $\int_{\R^d}|\nabla s_i|^2f_t \ge c\int_{\R^d}s_i^2 f_t$ we find
$$\tr\int_{\R^d}f_t \, \big(\hess\varphi_t\big)^2-\tr\Big(\int_{\R^d}f_t\hess\varphi_t\Big)^2\geq c\Big(\tr\Big(\int_{\R^d}f_t\hess\varphi_t\Big)^2-\tr\int_{\R^d}f_t\hess\varphi_t\Big).$$
This rewrites as $K(t)-M(t)\geq c(M(t)-J(t))$, which is equivalent to the desired inequality~\eqref{spectK}.

Substituting  inequality~\eqref{spectK} in~\eqref{5}, we find
$$\partial_t(e^{-2t}(J_2(t)-J(t)))\geq \frac{c}{1+c}e^{-2t}(K(t)-J(t)).$$
Integrating this inequality from $t$ to $\infty$, we obtain
$$J(t)-J_2(t)\geq \frac{c}{1+c}e^{2t}\int_t^{\infty}e^{-2s}(K(s)-J(s))ds.$$
Hence, using~\eqref{EFI},
\begin{align*}
\E\left(\frac{X+Y}{\sqrt{2}}\right)-\E(X)&=\int_0^{\infty}(J(t)-J_2(t))dt\\
&\geq\frac{c}{1+c}\int_0^{\infty}e^{2t}\int_t^{\infty}e^{-2s}(K(s)-J(s))dsdt\\
&=\frac{c}{2(1+c)}\int_0^{\infty}(1-e^{-2t})(K(t)-J(t))dt\\
&=\frac{c}{2(1+c)}\int_0^{\infty}(1-e^{-2t})(-\frac{1}{2}\partial_t(J(t)-d))dt\\
&=\frac{c}{2(1+c)}\int_0^{\infty}e^{-2t}(J(t)-d)dt.
\end{align*}
Applying Lemma~\ref{lem:intJ} we get
$$\E\left(\frac{X+Y}{\sqrt{2}}\right)-\E(X)\geq \frac{c}{4(1+c)}(\E(G)-\E(X)).$$
This ends the proof of Theorem~\ref{main}.

\section{Links with the Isotropic constant}

The isotropic constant of an isotropic log-concave random vector $X\sim f$ on $\R^d$ is defined by
$$L_X:= L_f := f(0)^{1/d}$$
This quantity  appears in several high-dimensional problems and a challenging open problem  in asymptotic convex geometry raised by Bourgain and known as the Slicing or Hyperplane conjecture,  is whether it is universally bounded (independently of $f$ and $d$). The best known bound is $L_f \le c\,  d^{1/4}$ for some universal constant $c>0$ (\cite{K1}).  See~\cite{Ba1,MP,K1,EK} for background, equivalent formulations and related results.


\begin{theorem}\label{L}
Let $X$ be an isotropic log-concave random vector in $\R^d$.
Assume that it satisfies an entropy jump with constant $\kappa \in (0,1)$:
$$
\E\left(\frac{X+Y}{\sqrt{2}}\right)-\E(X)\geq \kappa\,  (\E(G)-\E(X)).
$$
for $Y$  an independent copy of $X$. By Theorem~\ref{main}, this  holds with $\kappa=c/8$ if $f$ satisfies a spectral gap inequality~\eqref{spectralgap} with constant $c>0$. Then we have 
\begin{equation}\label{Lbound}
\displaystyle L_X \le  e^{2/\kappa}.
\end{equation}
\end{theorem}
Note that the bound also reads as 
$$L_X \le e^{16\,  {\rm c_p}(X)}$$
in terms of  ${\rm c_p}(X):={\rm c_p}(f)$, the Poincar\'e constant ~\eqref{spectralgap2} of $X\sim f$.

As a consequence, we see that the KLS conjecture (asserting that isotropic log-concave distributions satisfy a Poincar\'e inequality~\eqref{spectralgap}  with some  universal constant) implies the Hyperplance conjecture. In this direction, a better result is  known; indeed,   Eldan and Klartag~\cite{EK} recently proved that the \emph{variance conjecture} implies as well the   Hyperplane conjecture. The variance conjecture asserts that  inequality~\eqref{spectralgap} for the particular function  $s(x)=|x|^2 - \int |y|^2\, f(y)\, dy$ holds with a universal constant for every log-concave isotropic distribution $f$ on every dimension $d$.  However, it is worth noting  that unlike the Eldan-Klartag result, our estimate above holds at the level of an individual distribution $X$.

Theorem~\ref{L} was presented by the first named author in 2003 at a conference in Kiel and then expanded in a series of lectures in 2006 at the conference \emph{Phenomena in High Dimensions} at the I.H.P., as part of a more general program proposing a probabilistic viewpoint on the geometry of convex bodies in high dimensions. A similar program was also recently and independently proposed by Bobkov and Madiman (see e.g.~\cite{BoMa1, BoMa2}). 

Let us now explain the short and simple argument allowing us to pass  from the entropy jump to a bound on the isotropic constant. It relies on a classical rigidity property  of  isotropic log-concave distributions 
$X\sim f$ in $\R^d$, namely that  up to non-essential linear terms in $d$, we have $\log f(0)\simeq -\E(X)$. 
The following bound
\begin{equation}\label{fm}
-\log f(0)\leq \E(X)\leq -\log f(0)+d
\end{equation}
is for instance implicit in~\cite{FM} and the easy proof is as  follows. Write $f=e^{-\varphi}$ with $\varphi$ convex . For the lower bound use that $\int_{\R^d} x f(x)\, dx = 0$ together with  Jensen's inequality to get
$$-\log f(0)= \varphi(0) \le  \int_{\R^d} \varphi(x) \, f(x)\, dx = \E (X).$$
The upper bound combines the convexity of $f$ and an integration by parts as follows:
$$\E(X) -\varphi(0)=  \int_{\R^d} f(x) \big(\varphi(x) - \varphi(0)\big)\, dx \le \int_{\R^d} f(x) \nabla\varphi(x) \cdot x \, dx = d.$$

Let us mention that in the definition of $L_X$ and in the entropic bounds above 
, we can replace, up to numerical constants,  $f(0)$ by 
$||f||_{\infty}:=\sup_{\R^d}|f|$, since it is known (see~\cite{Fr}) that $||f||_{\infty}\leq e^df(0)$ for an isotropic log-concave distribution $f$ on $\R^d$.

To finish the proof of~\eqref{Lbound}, assume first that $X\sim f$ is symmetric, which means that $f$ is even. If $h$ denotes the density of $\frac{X+Y}{\sqrt{2}}$, then 
$h(x)=2^{\frac{d}{2}}\int_{\R^d} f(x-y)f(y)\, dy$.
It then follows from the log-concavity of $f$ that
$$h(0)=2^{\frac{d}{2}}\int_{\R^d} f(y)^2\, dy\geq 2^{\frac{d}{2}}\int_{\R^d} f(2y)f(0)\, dy=2^{-\frac{d}{2}}f(0).$$
Hence, using~\eqref{fm}  we have that
\begin{equation}\label{jumpS}
\E\left(\frac{X+Y}{\sqrt{2}}\right)\leq \frac{d}{2}\log 2\, -\log f(0) +d\le -\log f(0) +  \frac32 d.
\end{equation}
Let us now go back to the general case where $X\sim f$ is not necessarily symmetric, and consider $Y, X', Y'$ i.i.d. copies of $X$.  Then $\frac{X-X'}{\sqrt 2}, \frac{Y-Y'}{\sqrt 2}$ are symmetric log-concave isotropic random vectors in $\R^d$, independent  and identically distributed according to the  density  
$g(x)=2^{d/2}\int_{\R^d} f(x+y)f(y).$
It follows from the argument above that $g(0)\geq 2^{-d/2} f(0)$. 
%
Thus, by the Shannon-Stam inequality 
$\frac12\E(Z)+ \frac12 \E(U)\le \E\big(\frac{Z+U}{\sqrt{2}}\big)$ for $Z= (X+Y)/\sqrt2$ and 
$U=-(X'+Y')/\sqrt 2$ two independent random vectors, and the bound~\eqref{jumpS} obtained in the symmetric case, one gets
\begin{align*}
\E\left(\frac{X+Y}{\sqrt{2}}\right)&\leq \E\left(\frac{X+Y - X' - Y'}{\sqrt{4}}\right) =  \E\left(\frac{\frac{X-X'}{\sqrt 2} + \frac{Y-Y'}{\sqrt 2}}{\sqrt{2}}\right)
\\
&\leq -\log g(0) +\frac32 d \\
&\leq -\log f(0)+2d.
\end{align*}
On the other hand, the assumption on the Entropy jump implies 
$$(1-\kappa)\E(X)\leq \E\left(\frac{X+Y}{\sqrt{2}}\right)-\kappa\E(G) \le  \E\left(\frac{X+Y}{\sqrt{2}}\right)$$
since $\E(G)=\frac{d}{2}\log 2\pi e \ge 0$. Therefore, using again~\eqref{fm} we get that
$$(1-\kappa)\big(-\log f(0) \big) \leq -\log f(0)+ 2d. $$
This implies
$$\kappa \log f(0)\leq 2d.$$ 
and the desired bound~\eqref{Lbound}.

\subsection*{Acknowledgements}
The second author would like to thank his PhD advisor Dario Cordero-Erausquin for all his help and advice.

\end{document}